\documentclass[11pt,a4paper]{article}

\usepackage{anysize}
\marginsize{22mm}{14mm}{12mm}{25mm}

\usepackage[utf8]{inputenc}
\usepackage[T1]{fontenc}

\usepackage{bm}
\usepackage{amsmath}
\usepackage{amsfonts}
\usepackage{amsthm} 
\usepackage{mathrsfs}
\usepackage[affil-it]{authblk}

\usepackage[pdftex]{graphicx}
\usepackage{icomma}
\usepackage{enumitem}

\usepackage[title]{appendix}

\usepackage[ruled,vlined]{algorithm2e}

\usepackage{mdframed} 
\newmdenv{allfour}
\newmdenv[leftline=false,rightline=false]{topbot}
\newmdenv[topline=false,rightline=false]{leftbot}
\newmdenv[topline=false,leftline=false]{rightbot}
\newmdenv[topline=false,rightline=false,leftline=false]{bottom}

\usepackage[sort]{natbib}

\usepackage{multicol}

\newcommand{\tr}[1]{#1^{T}}

\newcommand{\cov}{\mathop{\text{Cov}}}

\newcommand{\e}{\mathop{\mathbb{E}}}

\newtheorem{theorem}{Theorem}[section]
\newtheorem{lemma}{Lemma}[section]
\newtheorem{definition}{Definition}[section]

\newtheorem{assumption}{Assumption}[section]
\newtheorem{example}{Example}[section]

\begin{document}

\setlength{\abovedisplayskip}{3pt}
\setlength{\belowdisplayskip}{3pt}
\setlength\parindent{0pt}

\title{Finite element approximation of non-Markovian random fields}
\author[1,2]{Mike PEREIRA}
\author[2]{Nicolas DESASSIS}
\affil[1]{{\small ESTIMAGES, Paris, France}}
\affil[2]{{\small Geostatistics team, MINES ParisTech - Geosciences, PSL Research University, Fontainebleau, France}}

\date{}
\maketitle

\begin{abstract}
In this paper, we present finite element approximations of a class of Generalized random fields defined over a bounded domain of $\mathbb{R}^d$ or a smooth $d$-dimensional Riemannian manifold ($d\ge 1$). An explicit expression for the covariance matrix of the weights of the finite element representation of these fields is provided and an analysis of the approximation error is carried out. Finally, a method to generate simulations of these weights while limiting computational and storage costs is presented.
\end{abstract}

\begin{tabular}{ll}
\textbf{Keywords} : & SPDE, Generalized random field, Riemannian manifold, Finite element method, \\
& Chebyshev approximation
\end{tabular}

\section{Introduction}
The "SPDE approach", as popularized by \citet{lindgren2011explicit}, consists in characterizing stationary continuous Markov random fields on $\mathbb{R}^d$ ($d\ge 1$) as solutions of stochastic partial differential equations (SPDE). This approach has two benefits:
\begin{itemize}
\item Discrete approximations of the solutions of these SPDEs, obtained using Galerkin methods such as the Finite Element method, are used in place of the original field in numerical computations. \citet{lindgren2011explicit} actually derived the expression of the precision matrix of the weights of the discrete representation of the solution, thus facilitating the use of this approach.
\item By tinkering with these same SPDEs, generalizations of stationary continuous Markov random fields on $\mathbb{R}^d$ can be defined on manifolds, and oscillating and even non-stationary random fields can be produced \citep{lindgren2011explicit, fuglstad2015exploring}.
\end{itemize}

This work aims at generalizing the SPDE approach to fields that are not continuously Markovian while still keeping the benefits mentioned above. First, the motivations for this work are detailed in order to point out the type of random fields that will be used throughout the developments. Then, an explicit formula for the covariance matrix of the weights of the finite element representation of such fields is provided, and an error analysis is carried out. Finally, an algorithm based on a Chebyshev polynomial approximation and allowing to compute simulations of these weights with a linear computational complexity is introduced.

\medskip
\textbf{Note} : This paper is a stub presenting the main results obtained by the authors. It is planned to be expanded.  

\section{Motivations}
\label{motiv}

Denote $\mathcal{W}$ the spatial Gaussian white noise on $\mathbb{R}^d$ defined on a complete probability space $(\Omega, \mathcal{A}, \mathbb{P})$. It can be seen as a Gaussian random measure satisfying:
\begin{equation*}
\forall A, B \in \mathcal{B}_B(\mathbb{R}^d), \quad \cov\left[\mathcal{W}(A), \mathcal{W}(B)\right]=\text{Leb}(A\cap B)
\end{equation*}
where $\mathcal{B}_B(\mathbb{R}^d)$ is the collection of bounded Borel sets of $\mathbb{R}^d$ and $\text{Leb}$ denotes the Lebesgue measure. 

\subsection{Solutions of stochastic partial differential equations}
\label{motiv:spde}
Let $g : \mathbb{R}_+ \rightarrow \mathbb{R}$ be a continuous, polynomially bounded function such that:
 $$\exists N>0, \quad \int_{\mathbb{R}^d} \vert g(\Vert \bm\omega\Vert^2)\vert^{-2}(1+\Vert \bm\omega\Vert^2)^{-N}d\bm\omega <\infty$$
And let $\mathcal{L}_g$ be the pseudo-differential operator defined over sufficiently regular functions of $\mathbb{R}^d$ by :
$$\mathcal{L}_g[.]=\mathcal{F}\left[\bm\omega \mapsto g(\Vert\bm\omega\Vert^2)\mathcal{F}[.](\bm\omega) \right]$$
Consider then the stochastic partial differential equation (SPDE) defined over $\mathbb{R}^d$ by \citep{vergara2018general}:
\begin{equation}
\mathcal{L}_gZ=\mathcal{W}
\label{spde_symbol}
\end{equation}
where $\mathcal{W}$ is a spatial Gaussian white noise and the equality is understood in the second order sense, i.e. both sides have same (generalized) covariance. The existence and uniqueness of a stationary solution of \eqref{spde_symbol} are guaranteed if $g$ is inferiorly bounded by the inverse of a strictly positive polynomial \citep{vergara2018general}.

Numerical solutions of \eqref{spde_symbol} on a triangulated domain $\mathcal{D}$ can be obtained using the Finite element method. A finite element approximation of the solution of \eqref{spde_symbol} is then built as:
$$Z(\bm x)=\sum\limits_{k=1}^nz_i\psi_i(\bm x), \quad \bm x\in\mathcal{D}$$
where $\lbrace \psi\rbrace_{1\le i\le n}$ is a family of deterministic basis functions and $\bm z=(z_1, \dots, z_n)^T$ is a vector of Gaussian weights. \citet{lindgren2011explicit} provided an expression for the precision matrix of these weights in the special case where $g$ is a real polynomial taking strictly positive values on $\mathbb{R}_+$, which corresponds to the case where $Z$ is a continuous Markov random field.

A first motivation for this work is to come up with numerical solutions of \eqref{spde_symbol} for a wider class of functions $g$, using the fact that within the framework presented above, the solution of \eqref{spde_symbol} is actually the Generalized random field defined by \citep{vergara2018general}:
$$Z=\mathcal{L}_{\frac{1}{g}}\mathcal{W}$$

\subsection{Generalized random fields}
\label{motiv:gen_rf}
Let $Z$ be an isotropic stationary real Gaussian random field on $\mathbb{R}^d$ with spectral density $f$. In particular, $f$ is a positive radial function of $\mathbb{R}^d$. \citet{lang2011fast} showed that then, $Z$ can be seen as a Generalized random field defined by:
$$Z=\mathcal{L}_{\sqrt{f}}\mathcal{W}=\mathcal{F}^{-1}\left[\bm\omega \mapsto \sqrt{f(\Vert\bm\omega\Vert^2)}\mathcal{F}[\mathcal{W}](\bm\omega)\right]$$
where $\mathcal{W}$ is once again a spatial Gaussian white noise. They used this characterization of Gaussian fields with spectral density $f$ to derive algorithms for the fast generation of samples of such fields over rectangular lattices of $\mathbb{R}^d$ using Fast Fourier transform.

A second motivation for this work is to combine this characterization of Gaussian fields with a given spectral density and the SPDE approach to come up with a way to generate samples of these fields on domains more complex  than lattices, namely irregular grids, general bounded domains of $\mathbb{R}^d$ and even Riemannian manifolds. This type of generalization was in particular exploited by \citet{lindgren2011explicit} in the particular case of continuous Markov random fields.

In the next section, the approximation of such Generalized random fields using the Finite element method is presented, and an error analysis on this approximation is carried out.

\section{Finite element approximation of Generalized random fields}

Le $d\in\mathbb{N}^*$ and let $\mathcal{D}$ be either a bounded (convex and polygonal) domain of $\mathbb{R}^d$, or a compact $d$-dimensional smooth Riemannian manifold. Denote $H=L^2(\mathcal{D})$, the separable Hilbert space of (real) square-integrable functions on $\mathcal{D}$. Denote $(., .)_H$ the inner product of $H$.

Let $L$ denote a densely defined, self-adjoint, positive semi-definite linear differential operator of second order, defined in a domain $\mathscr{D}(L)\subset H$ with Dirichlet boundary conditions on $\mathcal{D}$. $L$ is diagonalizable on a orthonormal basis $\lbrace e_j \rbrace_{j\in\mathbb{N}}$ of $H$. In particular, the eigenvalue-eigenfunction pairs of $L$, denoted $\lbrace (\lambda_j, e_j)\rbrace_{j\in\mathbb{N}}$, are arranged so that the eigenvalues $\lbrace \lambda_j \rbrace_{j\in\mathbb{N}}$ satisfy \citep{courant1966methods}:
$$0\le\lambda_1\le \lambda_2 \le \dots \le \lambda_j \le \dots, \quad \lim_{j\rightarrow +\infty}\lambda_j=+\infty$$

\subsection{Theoretical framework}

\subsubsection*{Differential operator on $H$}

For $\gamma : \mathbb{R}_+ \mapsto \mathbb{R}$, denote $H^\gamma=\left\lbrace\psi \in H : \sum_{j\in\mathbb{N}} \gamma(\lambda_j)^2(\psi_j, e_j)_H^2< \infty \right\rbrace$. Then we define the action of the differential operator $\gamma(L) : H^\gamma \rightarrow H$ on $H^\gamma$ by:
\begin{equation}
\gamma(L)\phi:=\sum\limits_{j\in\mathbb{N}}\gamma(\lambda_j)\left(\phi, e_j\right)_H e_j, \quad \phi\in H^\gamma
\end{equation} 
Remark that the subspace $H^\gamma$ is itself a Hilbert space with respect to the inner product $(., .)_\gamma$ and corresponding norm $\Vert . \Vert_\gamma$ defined by:
\begin{equation}
\begin{aligned}
(\phi,\psi)_\gamma =\left(\gamma(L)\phi, \gamma(L)\psi \right)_H=\sum_{j\in\mathbb{N}}\gamma(\lambda_j)^2(\phi, e_j)_H(\psi, e_j)_H 
\end{aligned}
\end{equation}

The following lemma gives a sufficient condition so that $H^\gamma=H$.
\begin{lemma} If $\gamma$ satisfies $\sum_{j\in\mathbb{N}}\gamma(\lambda_j)^2<\infty$ then $H^\gamma=H$.
\end{lemma}
\begin{proof}
$\forall \phi \in H$, $\sum_{j\in\mathbb{N}}(\phi, e_j)_H^2=\Vert\phi\Vert_H^2<\infty$, so in particular $\lim_{j\rightarrow +\infty} (\phi, e_j)_H^2 =0$. Therefore, $\exists J >0$ such that $j>J \Rightarrow (\phi, e_j)_H^2<1$, and so $\gamma(\lambda_j)^2(\phi, e_j)_H^2<\gamma(\lambda_j)^2$ which allows to conclude that the series $\sum_{j\in\mathbb{N}}\gamma(\lambda_j)^2(\phi, e_j)_H^2$ is convergent given that the series $\sum_{j\in\mathbb{N}}\gamma(\lambda_j)^2$ is convergent.
\end{proof}

In the particular case where $L=-\Delta$, where $\Delta$ denotes the Laplacian (or the Laplace-Beltrami operator) on $\mathcal{D}$, the operator $\gamma(L)$ satisfies the following property:

\begin{lemma} $\forall \phi\in H^\gamma$, $$\gamma(-\Delta)\phi=\mathcal{F}^{-1}\left[\bm w \mapsto \gamma(\Vert \bm w\Vert^2)\mathcal{F}[\phi](\bm\omega)\right]=\mathcal{L}_\gamma\phi$$
where $\mathcal{F}$ denotes the extension of the Fourier transform over $\mathcal{D}$.
\end{lemma}
Details and proof of this lemma are provided in Appendix \ref{sec::four}. In particular, the motivational cases presented in Section \ref{motiv} correspond to the case where $\gamma=1/g$ for Section \ref{motiv:spde} and $\gamma=\sqrt{f}$ for Section \ref{motiv:gen_rf}.

\subsubsection*{Generalized random fields of $H$}

Denote $\mathcal{W}$ the spatial Gaussian white noise on $\mathcal{D}$ defined on a complete probability space $(\Omega, \mathcal{A}, \mathbb{P})$. A characterization of $\mathcal{W}$ based on the Hilbert space $H$ is given by the following lemma.

\begin{lemma}
Let $\lbrace \tilde{\xi}_j \rbrace_{j\in\mathbb{N}}$ be a sequence of independent, standard normally-distributed random variables. Then, the linear functional defined over $H$ by :
$$\phi \in H \mapsto \sum_{j\in\mathbb{N}} \tilde{\xi}_j(\phi, e_j)_H$$
is a Gaussian white noise (which is also denoted $\mathcal{W}$). In particular, it satisfies:
\begin{equation*}
\forall \phi, \psi \in H, \quad \cov\left[\mathcal{W}(\phi), \mathcal{W}(\psi)\right]=(\phi, \psi)_H
\label{cov_wn}
\end{equation*}
\label{lm::wn}
\end{lemma}
\begin{proof}
Denote $\Xi : \phi \in H \mapsto \sum_{j\in\mathbb{N}} \tilde{\xi}_j(\phi, e_j)_H$. For an integer $m\ge 1$, take a set $\phi_1, \dots, \phi_m \in H$ of linearly independent elements of $H$ and denote $X=(\Xi(\phi_1), \dots, \Xi(\phi_m))^T$. Let's show that $X$ is a Gaussian vector. Indeed, the characteristic function of this vector is: 
\begin{align*}
\Phi(w)=\e\left[e^{i\tr{w}X}\right]
=\e\left[\exp\left({i\sum\limits_{k=1}^m w_k\Xi(\phi_k)}\right)\right]
=\e\left[\exp\left({i\sum_{j\in\mathbb{N}}\tilde{\xi}_j\sum\limits_{k=1}^m w_k  (\phi_k, e_j)_H}\right)\right]
\end{align*}
Using the fact that the $\xi_j$ are independent standard Gaussian variables yields:
\begin{align*}
\Phi(w)&=\prod\limits_{j\in\mathbb{N}}\e\left[\exp\left({i\tilde{\xi}_j\sum\limits_{k=1}^m w_k  (\phi_k, e_j)_H}\right)\right]
=\prod\limits_{j\in\mathbb{N}}\exp\left(-\frac{1}{2}\left(\sum\limits_{k=1}^m w_k  (\phi_k, e_j)_H\right)^2\right)\\
&=\exp\left(-\frac{1}{2}\sum_{j\in\mathbb{N}}\sum\limits_{k=1}^m\sum\limits_{l=1}^m w_k  (\phi_k, e_j)_H w_l(\phi_l, e_j)_H\right)
=\exp\left(-\frac{1}{2}\sum\limits_{k=1}^m\sum\limits_{l=1}^m w_kw_l  \sum_{j\in\mathbb{N}}(\phi_k, e_j)_H (\phi_l, e_j)_H\right)\\
&=\exp\left(-\frac{1}{2}\sum\limits_{k=1}^m\sum\limits_{l=1}^m w_kw_l (\phi_k, \phi_l)_H\right)=\exp\left(-\frac{1}{2}w^TKw\right)
\end{align*}
where $K$ is the positive definite matrix whose entries are $K_{kl}=(\phi_k, \phi_l)_H$, $1\le k,l\le m$. Therefore, we can conclude that $X$ is a Gaussian vector, and in particular, by definition of $K$, \eqref{EF_def} is satisfied by $\Xi$. \\
Hence, $\Xi$ can be identified to the spatial Gaussian white noise on $\mathcal{D}$ by defining for $ A \in \mathcal{B}(\mathcal{D})$ the measure $\Xi(A):=\Xi(\bm 1_A)$ where $\bm 1_A\in H$ is the indicator function of the set $A$.
\end{proof}

Denote $L^2(\Omega, H)$ the Hilbert space of $H$-valued random variables satisfying $\e[\Vert f\Vert^2_H]<\infty$ and equipped with the scalar product $(f, g)_{L^2(\Omega, H)}=\e\left[(f, g)_H\right]$. The next result introduces a class of Generalized random fields defined through the white noise that can be identified with elements of $L^2(\Omega, H)$.

\begin{definition}
Let $\gamma : \mathbb{R}_+ \mapsto \mathbb{R}$ such that :
\begin{equation}
\sum_{j\in\mathbb{N}} \gamma(\lambda_j)^2 <\infty
\end{equation}
Then, $\gamma(L)\mathcal{W}$ denotes the Generalized random field defined by:
\begin{equation}
(\gamma(L)\mathcal{W})[\phi] := \mathcal{W}(\gamma(L)[\phi]), \quad \phi\in H
\end{equation}
\end{definition}

\begin{lemma}
$\gamma(L)\mathcal{W}$ can be identified to an element $Z\in L^2(\Omega, H)$  defined by:
\begin{equation*}
Z=\sum\limits_{j\in\mathbb{N}}\tilde{\xi}_j\gamma(\lambda_j) e_j
\end{equation*}
for a sequence $\lbrace \tilde{\xi}_j \rbrace_{j\in\mathbb{N}}$ of independent standard normally-distributed random variables, through the linear functional of $H$ : $(\gamma(L)\mathcal{W})(\phi)=(Z, \phi)_H$, $\phi\in H$.
\end{lemma}
\begin{proof}
Clearly, $Z$ is an element of $L^2(\Omega, H)$ given that $$\Vert Z\Vert_{L^2(\Omega, H)}^2 =\e\left[\Vert Z\Vert_{H}^2\right]  =  \sum_{j\in\mathbb{N}} \gamma(\lambda_j)^2 <\infty$$
Let's now show that the linear functional $\phi\in H \mapsto (Z, \phi)_H$ is equal to $\gamma(L)\mathcal{W}$. Indeed, $\forall \phi\in H$,
\begin{align*}
(\gamma(L)\mathcal{W})(\phi)=\mathcal{W}(\gamma(L)\phi)=\mathcal{W}\left(\sum\limits_{j\in\mathbb{N}}\gamma(\lambda_j)\left(\phi, e_j\right)_H e_j\right)
\end{align*}
So, using Lemma \ref{lm::wn}, 
\begin{align*}
(\gamma(L)\mathcal{W})(\phi)=\sum\limits_{j\in\mathbb{N}}\tilde{\xi}_j\gamma(\lambda_j)\left(\phi, e_j\right)_H=(Z, \phi)_H
\end{align*}
\end{proof}

From now on, Generalized random fields $Z$ of the form $\gamma(L)\mathcal{W}$ will be directly identified with their $L^2(\Omega, H)$ representation, and we will write:
\begin{equation}
Z=\gamma(L)\mathcal{W}=\sum\limits_{j\in\mathbb{N}}\tilde{\xi}_j\gamma(\lambda_j) e_j
\label{gen_field_def}
\end{equation}
where $\lbrace \tilde{\xi}_j \rbrace_{j\in\mathbb{N}}$ is a sequence of independent, standard normally-distributed random variables. In the next section, the simulation of such random fields through a finite element scheme is presented.

\subsection{Finite element approximation of a Generealized random field}
Let $\mathcal{T}_h$ denote a triangulation of $\mathcal{D}$ with mesh size $h$ and $\Psi=\lbrace \psi_k \rbrace_{1\le k \le n}$ a family of linearly independent basis functions associated with $\mathcal{T}_h$ such that $\Psi \subset H$. 

Denote $V_h \subset H$ the linear span of $\Psi$, which is a $n$-dimensional space. Denote $\lbrace f_{j,h} \rbrace_{1\le j \le  n}$ an orthonormal basis of $V_h$ with respect to the scalar product $(., .)_H$. The discretization of the operator $L$ on $V_h$ is denoted $L_h$ and is defined by:
\begin{equation}
\begin{aligned}
L_h : V_h \rightarrow V_h, \quad \psi \mapsto L_h\psi=\sum\limits_{j=1}^n \left(L \psi, f_{j,h}\right)_Hf_{j,h}
\end{aligned}
\end{equation}

Let $\bm C$ and $\bm G$ be the matrices defined by:
\begin{align*}
\bm C &= \left[( \psi_i, \psi_j)_H\right]_{1\le i,j\le n} \\
\bm G & =\left[(L\psi_i, \psi_j)_H\right]_{1\le i,j\le n}
\end{align*}
$\bm C$ is a symmetric positive definite matrix called Mass matrix and $\bm G$ is a symmetric positive semi-definite matrix called stiffness matrix. Denote $\bm C ^{1/2}$ the symmetric positive definite square root of $\bm C$, and $\bm C^{-1/2}$ its inverse.

\begin{lemma}
$L_h$ is diagonalizable on $V_h$ and its eigenvalues are those of the matrix $\bm S$ defined by:
$$\bm S = \bm C^{-1/2}\bm G\bm C^{-1/2}$$
In particular, the application:
$$E: \bm v\in\mathbb{R}^n \mapsto \sum\limits_{j=1}^n [\bm C^{-1/2}\bm v]_j\psi_j$$
is an isometric isomorphism that maps the eigenvectors of $\bm S$ to the eigenfunctions of $L_h$. 
\end{lemma}

\begin{proof}
Take $\lambda$ an eigenvalue of $\bm S$ and denote $\bm v \neq \bm 0$ an associated eigenvector. Then, $\bm S \bm v =\bm C^{-1/2}\bm G\bm C^{-1/2}\bm v= \lambda \bm v$ and so, $\bm G \bm u = \lambda\bm C \bm u$ where $\bm u=\bm C^{-1/2}\bm v$. Hence, $\forall k \in[\![1, n]\!]$, $\sum_{j} \left(L \psi_k, \psi_j\right)_Hu_j=\lambda\sum_{j} \left(\psi_k, \psi_j\right)_Hu_j$, which, using the fact that $L$ is self-adjoint, gives 
\begin{equation}
\forall k \in[\![1, n]\!], \quad \left( L E(\bm v), \psi_k\right)_H= \lambda\left( E(\bm v), \psi_k\right)_H
\label{eq_psi}
\end{equation}
In particular, given that $\Psi$ is also a basis of $V_h$, we denote $\bm A=[a_{ij}]_{1\le i,j\le n}$ the invertible change-of-basis matrix between $\Psi$ and $\lbrace f_{j,h} \rbrace_{1\le j \le  n}$. 
Then \eqref{eq_psi} can be written,
$$\bm A \begin{pmatrix}
\left( L E(\bm v), f_{1,h}\right)_H \\ \vdots \\ \left( L E(\bm v), f_{n,h}\right)_H
\end{pmatrix}=\lambda\bm A \begin{pmatrix}\left(E(\bm v), f_{1,h}\right)_H \\ \vdots \\ \left( E(\bm v), f_{n,h}\right)_H
\end{pmatrix}$$
And therefore, $\forall j\in [\![1, n]\!]$, $\left( L E(\bm v), f_{j,h}\right)_H=\lambda\left(  E(\bm v), f_{j,h}\right)_H$. And so, given that $E(\bm v)\in V_h$, 
\begin{align*}
L_h E(\bm v) = \sum\limits_{j} \left(L E(\bm v), f_{j,h}\right)_H f_{j,h}=\sum\limits_{j} \lambda\left(E(\bm v), f_{j,h}\right)_H f_{j,h}=\lambda E(\bm v)
\end{align*}
Therefore $\lambda$ is an eigenvalue of $L_h$ and $E$ maps the eigenvectors of $\bm S$ to the eigenfunctions of $L_h$.

$E$ is an isometry: Indeed, $\forall \bm v\in\mathbb{R}^n$, $\Vert E(\bm v)\Vert_H^2=(E(\bm v), E(\bm v))_H=\sum_{i,j} [\bm C^{-1/2}\bm v]_i(\psi_i, \psi_j)_H [\bm C^{-1/2}\bm v]_j=\left(\bm C^{-1/2}\bm v\right)^T \bm C \bm C^{-1/2}\bm v=\bm v^T\bm v$. Consequently, given that it is also linear, $E$ is injective. And finally, using the rank–nullity theorem, $E$ is bijective (as an injective application between two vector spaces with same dimension).
\end{proof}

Let's denote $\lbrace \lambda_{j,h}\rbrace_{1\le j\le n} \subset \mathbb{R}_+$ the eigenvalues of $L_h$ (or equivalently $\bm S$) and  $\lbrace e_{j,h}\rbrace_{1\le j\le n}$ the associated orthonormal eigenfunctions, obtained by mapping by $E$ the orthonormal eigenbasis $\bm V =(\bm v_1| \dots | \bm v_n)$ of $\bm S$. In particular,
\begin{equation}
\bm S =\bm V \begin{pmatrix}
\lambda_{1,h} & & \\
 & \ddots & \\
 & & \lambda_{n,h}
\end{pmatrix}\bm V^T, \quad \text{ and } \quad  e_{j,h}=E(\bm v_j), 1\le j\le n
\end{equation}

Take $\gamma : \mathbb{R_+} \mapsto \mathbb{R}$. The discretization of the operator $\gamma(L)$ on $V_h$, denoted $\gamma(L_h)$, can now be defined as:
$$ \gamma(L_h) : V_h \rightarrow V_h, \quad \psi \mapsto \gamma(L_h)\psi := \sum\limits_{j=1}^n\gamma(\lambda_{j,h})\left(\psi, e_{j,h}\right)_H e_{j,h}$$

\begin{definition}
Let $\mathcal{W}_h$ be a $V_h$-valued random variable defined by :
$$\mathcal{W}_h=\sum\limits_{j=1}^n \tilde\xi_i e_{j,h}$$
where $\tilde{\bm \xi}=(\xi_1, \dots, \xi_n)^T$ is a vector whose entries are independent zero-mean (standard) normally-distributed random variables. Then, $\mathcal{W}_h$ is called white noise on $V_h$.
\end{definition}

\begin{lemma}
Let $\mathcal{W}_h$ be white noise in $V_h$. Then $\mathcal{W}_h$ can be written $\mathcal{W}_h=\sum_{j=1}^n \xi_i\psi_i$
where $\bm \xi=(\xi_1, \dots, \xi_n)^T \sim \mathcal{N}(\bm 0, \bm C^{-1})$.
\end{lemma}
\begin{proof}
Using the linearity of $E$, $\mathcal{W}_h\in V_h$ can be written $\mathcal{W}_h=\sum_{j=1}^n \tilde \xi_i E(\bm v_i)=E\left(\sum_{j=1}^n \tilde \xi_i\bm v_i\right)=E\left( \bm V\tilde{\bm \xi}\right)$ where $\tilde{\bm \xi} \sim \mathcal{N}(\bm 0,\bm I)$. But also, in the basis $\Psi$, we get $\mathcal{W}_h=\sum_{j=1}^n  \xi_i \psi_i=E\left(\bm C^{1/2}\bm \xi\right)$. In particular, using the fact that $E$ is bijective, $\bm\xi=\bm C^{-1/2}\bm V\tilde{\bm\xi}$ which proves the result.
\end{proof}

\begin{definition}
The $V_h$-valued random variable $Z_h$ defined by:
\begin{equation}
Z_h=\gamma(L_h)\mathcal{W}_h
\label{EF_def}
\end{equation}
is called finite element approximation of the generalized random field $Z$ defined by \eqref{gen_field_def}.
\end{definition}

\begin{theorem}
The finite element approximation $Z_h$ defined by \eqref{EF_def}
can be decomposed as: $$Z_h=\sum_{j=1}^n z_i \psi_i$$
for some multi-normally distributed weights $\bm z = (z_1, \dots, z_n)^T$ with mean $0$ and covariance matrix:
\begin{equation}
\bm\Sigma_{\bm z}=\bm C^{-1/2}\gamma^2(\bm S)\bm C^{-1/2} \text{ where } \gamma^2(\bm S):=\bm V \left(\begin{smallmatrix}
\gamma(\lambda_{1,h})^2 & & \\
 & \ddots & \\
 & & \gamma(\lambda_{n,h})^2
\end{smallmatrix}\right)\bm V^T
\label{cov_mat_EF}
\end{equation}
\label{th:weights}
\end{theorem}

\begin{proof}
Notice that $Z_h=\sum_{j=1}^n z_i \psi_i=E\left(\bm C^{1/2}\bm z\right)$. But also, 
\begin{equation*}
Z_h=\gamma(L_h)\mathcal{W}_h=\sum_{j=1}^n \gamma(\lambda_i)\tilde \xi_i E(\bm v_i)=E\left(\sum_{j=1}^n \gamma(\lambda_i)\tilde \xi_i \bm v_i\right)=E\left(\bm V\left(\begin{smallmatrix}
\gamma(\lambda_{1,h}) & & \\
 & \ddots & \\
 & & \gamma(\lambda_{n,h})
\end{smallmatrix}\right)\tilde{\bm \xi}\right)
\end{equation*}
Therefore, given that $E$ is bijective, $\bm z=\bm C^{-1/2}\bm V\left(\begin{smallmatrix}
\gamma(\lambda_{1,h}) & & \\
 & \ddots & \\
 & & \gamma(\lambda_{n,h})
\end{smallmatrix}\right)\tilde{\bm \xi}$ where $\tilde{\bm \xi} \sim \mathcal{N}(\bm 0, \bm I)$, which proves the result.
\end{proof}

Theorem \ref{th:weights} provides an explicit expression for the the covariance matrix of the weights of the finite element approximation of a field defined by \eqref{gen_field_def}. This expression agrees in particular with the expression of the precision matrix of those same weights exposed in \citep{lindgren2011explicit} for the particular case of continuous Markovian fields on $\mathbb{R}^d$. Given the generality of domains $\mathcal{D}$ (convex bounded or Riemannian manifold), differential operators $L$ and functions $\gamma$, a wide class of fields are now open to study.

Note in particular that Riemannian manifolds provide a generic framework for the study of non-stationary fields on a manifold. Indeed, stationary fields on a Riemannian manifold $M$ equipped with a metric $u$ can be seen as non-stationary random fields on the manifold $M$ with locally-varying anisotropies defined by $u$. 

In the next section, an error bound between a generalized random field $Z$ defined by \eqref{gen_field_def} and its finite element approximation defined by \eqref{EF_def} is provided, using the same framework as in \citep{bolin2017numerical}.

\subsection{Error analysis of the finite element approximation}

First, we recall the notations used in this paper.\\
Let $(V_h)_{h\in]0,1]}$ be a family of finite element spaces indexed by a mesh size $h$ over a domain $\mathcal{D}\subset\mathbb{R}^d$. Let's denote $N_h=\text{dim}(V_h)$ be the number of basis functions associated with the triangulation of $\mathcal{D}$ with mesh size $h$. \\
Let $L$ denote a densely defined, self-adjoint, positive semi-definite linear differential operator of second order defined on a subset of $H=L^2(\mathcal{D})$, and let $L_h$ denote its discretization over $V_h$. Let $\lbrace \lambda_{j} \rbrace_{j\in\mathbb{N}}$ and $\lbrace \lambda_{j,h} \rbrace_{1\le j\le N_h}$ be the eigenvalues of $L$ and $L_h$, listed in non-decreasing order.
\\
Let $\gamma : \mathbb{R}_+ \rightarrow \mathbb{R}$.

The following assumptions are considered to derive an error bound between a Generalized random field $Z$ defined by \eqref{gen_field_def} and its finite element approximation defined by \eqref{EF_def}.

\begin{assumption}[Growth of the eigenvalues of $L$]
\label{assum:lambda}
There exists tree constants $\alpha >0$, $c_\lambda>0$ and $C_\lambda>0$ such that:
$$\forall j\in\mathbb{N}, \quad \lambda_j >0 \Rightarrow c_\lambda j^\alpha \le \lambda_j \le C_\lambda j^\alpha$$
\end{assumption}

\begin{assumption}[Derivable of $\gamma$]
\label{assum:hold}
$\gamma$ is derivable on $\mathbb{R}_+$, and there exist $C_{\text{Deriv}}>0$ and $a\ge 0$ such that:
$$ \forall x > 0, \quad \vert\gamma'(x)\vert \le \frac{C_{\text{Deriv}}}{x^a}$$
\end{assumption}

\begin{assumption}[Asymptotic behavior of $\gamma$]
\label{assum:gamma}
There exists a constant $\beta >0$ such that $\gamma$ satisfies $\vert\gamma(\lambda)\vert =\mathop{\mathcal{O}}\limits_{\lambda \rightarrow+\infty}\left( \lambda^{-\beta}\right)$, i.e. 
$$\exists C_{\gamma} > 0, \exists R_{\gamma}>0,  \quad \lambda \ge R_{\gamma} \Rightarrow \vert\gamma(\lambda)\vert \le C_{\gamma} \lambda^{-\beta}$$
\end{assumption}

\begin{assumption}[Dimension of the finite element space]
\label{assum:Nh}
There exists two constants $\tilde d >0$, $C_{\text{FES}}>0$  such that:
$$N_h=\text{dim}(V_h)= C_{\text{FES}}h^{-\tilde d}$$
\end{assumption}

\begin{assumption}[Mesh size]
\label{assum:h}
The mesh size $h$ shall satisfy:
$$h \le \left(\frac{1}{C_{\text{FES}}}\left\lceil\left(\frac{R_{\gamma}}{c_\lambda}\right)^{1/\alpha}\right\rceil\right)^{-1/\tilde d}$$
where $C_{\text{FES}}$, $R_\gamma$, $\alpha$ and $c_\lambda$ are the constants defined in Assumptions \ref{assum:Nh}, \ref{assum:gamma} and \ref{assum:lambda}.\\
Consequently, following Assumptions \ref{assum:Nh} and \ref{assum:lambda}, for all $j \ge N_h$, $\lambda_j \ge  R_\gamma$.
\end{assumption}

\begin{assumption}[Eigenvalues and eigenvectors of $L_h$]
\label{assum:Lh}
There exists constants $C_1, C_2 >0$, $h_0\in ]0,1[$ and exponents $r, s > 0$ and $q>1$ such that the eigenvalues $\lbrace \lambda_{j,h}\rbrace_{1\le j\le N_h}$ and the eigenvectors $\lbrace e_{j,h}\rbrace_{1\le j\le N_h}$ of the discretisation $L_h$ of the operator $L$ onto the finite element space associated with a triangulation of mesh size $h$ satisfy:
$$0 \le \lambda_{j,h} - \lambda_j \le C_1h^r\lambda_j^q$$
$$\Vert e_{j,h} - e_j\Vert_H^2 \le C_2 h^{2s}\lambda_j^q$$
for all $h\in ]0,h_0[$ and for all $j\in [\![1, N_h]\!]$
\end{assumption}

Following the notations of the previous sections, let $Z$ and $Z_h$ be the random fields defined by: 
\begin{equation}
Z=\gamma(L)\mathcal{W}=\sum\limits_{j\in\mathbb{N}}\gamma(\lambda_j)\tilde{\xi}_j e_j
\end{equation}
and
\begin{equation}
Z_h=\gamma(L_h)\mathcal{W}_h=\sum\limits_{j=1}^{N_h}\gamma(\lambda_{j,h})\tilde{\xi}_j e_{j,h}
\end{equation} 
The expected error defined by :
\begin{equation}
\Vert Z - Z_h \Vert_{L^2(\Omega; H)}=\sqrt{\e\left[\Vert Z - Z_h \Vert_{H}^2\right]}
\end{equation}
is bounded using the following result.

\begin{theorem}
If $V_h$, $\gamma$, $L$ and $L_h$ satisfy Assumptions (1-6), and if the growth of eigenvalues $\alpha$ defined in Assumption \ref{assum:lambda} satisfies:
$$\max\left(\frac{1}{2\beta}, \frac{1}{2a}\right) \le \alpha \le \min\left(\frac{2s}{dq}, \frac{r}{dq}\right)$$
where the constants $a, \beta, s, \tilde d, q, r$ are defined  in Assumptions (1-6), then, for $h$ sufficiently small (cf. Assumption \ref{assum:h}), the $L^2(\Omega, H)$ error between the generalized random field $Z$ defined by \eqref{gen_field_def} and its finite element approximation defined by \eqref{EF_def} is bounded by:
\medskip
\begin{equation}
\Vert Z_{N_h} - Z_h \Vert_{L^2(\Omega; H)} \le M h^{\min\left(s-dq\alpha/2, r-dq\alpha, \tilde d(\alpha\beta-1/2),  \tilde d(\alpha a-1/2) \right)}
\end{equation}
\medskip
where $M$ is a constant independent of $h$.
\label{th:err_fem}
\end{theorem}
The proof of this theorem is provided in Appendix \ref{proof:th_err_fem} and is an adaptation of proof of Theorem 2.10 of \citep{bolin2017numerical}, which provide a bound for the same approximation error for the particular case where $\gamma : x \mapsto 1/x^\beta$, $\beta \in ]0,1[$, and $L$ is a positive definite.

\begin{example} In the particular case where $L=-\Delta$, $L$ is a strongly elliptic differential operator of order $2$. Weyl's law \citep{weyl1911asymptotische} gives an exponent $\alpha$ for which Assumption \ref{assum:lambda} is satisfied:
$$\alpha=\frac{d}{2}$$
Moreover, if we assume that the finite element spaces are quasi-uniform and are composed of continuous piecewise polynomial functions of degree $p\ge 1$, then Assuptions \ref{assum:Nh} and \ref{assum:Lh} are satisfied for the exponents \citep{bolin2017numerical, strang1973analysis}:
$$r=2(p-1), \quad s=\min\lbrace p+1, 2(p-1)\rbrace, \quad q=\frac{p+1}{2}$$
\end{example}

\section{Finite element simulations}

Simulation of the finite element approximation \eqref{EF_def} of fields satisfying \eqref{gen_field_def} comes down to the simulation of the weights $\bm z=(z_1, \dots, z_n)^T$ of the decomposition onto the basis functions of the finite element space. These weights are Gaussian, with covariance matrix \eqref{cov_mat_EF}.

A straightforward method to generate admissible samples $\bm z$ would consist in computing:
$$\bm z = \bm C^{-1/2}\bm V \left(\begin{smallmatrix}
\gamma(\lambda_{1,h}) & & \\
 & \ddots & \\
 & & \gamma(\lambda_{n,h})
\end{smallmatrix}\right)\bm\epsilon$$
where $\bm\epsilon$ is a vector with $n$ independent standard Gaussian components. But this method supposes that the matrix $\bm S$ has been diagonalized and that its eigenvalues $\lbrace \lambda_{j,h}\rbrace$ and its eigenvectors $\bm V$ have been stored. Both operations being tremendously costly ($\mathcal{O}(n^3)$ for the diagonalization and $\mathcal{O}(n^2)$ for the storage), another method that doesn't involve them is highly preferable.

We propose to rather simulate the weights using the Chebyshev simulation algorithm presented in \citep{pereira2018efficient}. Indeed, the algorithm can provides vector samples with covariance matrix \eqref{cov_mat_EF} with a computational complexity that is linear with the number of non-zeros $\bm S$ (which is small as $\bm S$ is a sparse matrix) and an order of approximation which is fixed using a criterion that ensures that the statistical properties of the output are valid. Concerning the storage needs, only a matrix as sparse as $\bm S$ needs to be stored as the algorithm relies on matrix-vector multications.  This algorithm is reminded below.

\begin{figure}[h]
\begin{topbot}[innertopmargin=2ex,innerbottommargin=1ex]
\textbf{Algorithm}: Simulation of the weights using Chebyshev approximation \citep{pereira2018efficient}\\
\noindent\makebox[\linewidth]{\rule{\textwidth}{0.5pt}}\\
\textbf{Require}: An order of approximation $K\in\mathbb{N}$. A vector of $n$ independent standard Gaussian components $\bm\epsilon$.\\
\textbf{Output}: A vector $\bm z$ with covariance matrix (approximately equal to) \eqref{cov_mat_EF}.
\vspace{-1ex}
\begin{enumerate}
\item Find an interval $[a, b]$ containing all the eigenvalues of $\bm S$ (using for instance the Gershgorin cirle theorem).
\item Compute a polynomial approximation $\mathrm{P}$ of the function $\gamma$ over $[a, b]$, by truncating its (shifted) Chebyshev series at order $K$.\\
The coefficients of the development in Chebyshev series of $\gamma$ are computed by Fast Fourier Transform.
\item Compute iteratively the product $\bm u = \mathrm{P}(\bm S)\bm\epsilon$ using the recurrence relation satisfied by the Chebyshev polynomials.
\item The simulated field is given by: $\bm z = \bm C^{-1/2} \bm u$
\end{enumerate}
\end{topbot}
\end{figure}

\section{Conclusion}
In this work we provided an explicit expression for weights of the finite element approximation of a Generalized random field defined over of domain $\mathcal{D}$ consisting of a bounded domain of $\mathbb{R}^d$ or a $d$-dimensional smooth Riemannian manifold,  by $Z=\gamma(L)\mathcal{W}$ where $\gamma : \mathbb{R}_+ \rightarrow \mathbb{R}$, $L$ is a second-order self-adjoint positive definite differential operator, and $\mathcal{W}$ is a Gaussian white noise. An error bound on this approximation was derived and an algorithm for fast and efficient sampling of this field was exposed.


\nocite{*}
\bibliographystyle{apalike}
\bibliography{biblio.bib}
\addcontentsline{toc}{section}{References}

 \begin{appendices}

\section{Laplacian and Fourier transform}
\label{sec::four}

\subsection{Multivariate Fourier series and transform}
Let $g(\bm x)$ be a $2\pi$-periodic function of $\mathbb{R}^d,$ i.e. $g$ is $2\pi$-periodic with respect to each variable $x_1, \dots, x_d$ and suppose that $g\in L^2([-\pi,\pi]^d)$. Then $g$ can be represented as the limit on $L^2$ of its Fourier series $\mathcal{S}_f[g]$ defined by \citep{OSBORNE2010115}:
$$\mathcal{S}_F[g](\bm x)=\sum\limits_{\bm j\in\mathbb{Z}^d}c_{\bm j}(g)e^{i\bm j^T\bm x}$$
where the coefficients $c_{\bm j}(g)$ are given by:
$$c_{\bm j}(g)=\frac{1}{(2\pi)^d}\int_{[-\pi,\pi]^d}e^{-i\bm j^T\bm x}g(\bm x)d\bm x$$
The Fourier transform of $g$ is then defined from its Fourier series as a train of impulses (corresponding to the Fourier transform of each term of Fourier series): 
\begin{equation}
\mathcal{F}[g](\bm\omega)=(2\pi)^d\sum\limits_{\bm j\in\mathbb{Z}^d}c_{\bm j}(g)\delta_{\bm j}(\bm\omega)
\end{equation}
where $\delta_{\bm j}(\bm\omega)=\prod\limits_{j=1}^d\delta_{j_k}(\omega_k)$ and $\delta_\lambda(.)$ is the Dirac impulse at $\lambda$. 

\subsection{Proof}

Let's first consider the case $\mathcal{D}=[0,\pi]^d$ and denote $H=L^2(\mathcal{D})$. The eigenvalues and eigenfunctions of $-\Delta$ on $\mathcal{D}$ with Dirichlet boundary conditions are given for $\bm j\in\mathbb{N}^d$ by \citep{grebenkov2013geometrical}:
\begin{equation}
e_{\bm j}=\left(\frac{2}{\pi}\right)^{d}\prod\limits_{k=1}^d\sin(j_kx), \quad \lambda_{\bm j}= \sum\limits_{k=1}^d j_k^2
\label{eig_lab}
\end{equation}

\begin{lemma}
Let $f\in L^2([0,\pi]^d)$ and define $\tilde{f} : \mathbb{R}^d \rightarrow \mathbb{R}$ by:
\begin{equation}
\left\lbrace\begin{array}{ll}
\tilde{f}(\bm x)=f(\bm x), & \forall \bm x\in[0,\pi]^d \\
\tilde f(x_1,\dots, -x_k,\dots,x_d)=-\tilde f(x_1,\dots, x_k,\dots,x_d) & \forall \bm x\in\mathbb{R}^d, 1\le k\le d \\
\tilde f(\bm x+2\pi\bm n)=\tilde f(\bm x), & \forall \bm x\in\mathbb{R}^d, \bm n\in\mathbb{Z}^d
\end{array}\right.
\label{transfo_per}
\end{equation}
Then the coefficients $c_{\bm j}(\tilde{f})$ of the Fourier series of $\tilde f$ satisfy $\forall \bm j\in\mathbb{Z}^d$:
\begin{equation}
c_{\bm j}(\tilde{f})=\frac{1}{(2i)^d}\varepsilon(\bm j)(f, e_{\vert\bm j\vert})_H
\end{equation}
where $\varepsilon(\bm j)=\prod_{k=1}^d\textnormal{sign}(j_k)$, $\vert\bm j\vert:=(\vert j_1\vert, \dots, \vert j_1\vert)^T\in\mathbb{N}^d$ and $e_{\vert \bm j\vert}$ is an eigenfunction of $-\Delta$ on $[0,\pi]^d$ with Dirichlet boundary conditions, as defined in \eqref{eig_lab}. (Note : $\textnormal{sign} : x\in\mathbb{R} \mapsto 1$ if $x\ge 0$, $-1$ otherwise)\\
In particular, the Fourier series of $\tilde{f}$ (restricted to $[0,\pi]^d$) coincides (up to a normalization constant) with the development of $f$ in the eigenbasis of the Laplacian.
\end{lemma}

\begin{proof}
\begin{align*}
c_{\bm j}(\tilde{f})=\frac{1}{(2\pi)^d}\int_{[-\pi,\pi]^d}e^{-i\bm j^T\bm x}\tilde{f}(\bm x)d\bm x
=\frac{1}{(2\pi)^d}\int_{[-\pi,\pi]^{d-1}}e^{-i\sum\limits_{k=1}^{d-1} j_kx_k}\int_{[-\pi,\pi]}e^{-ij_dx_d}\tilde{f}(\bm x)d\bm x
\end{align*}
And,
\begin{align*}
\int_{[-\pi,\pi]}e^{-ij_dx_d}\tilde{f}(\bm x)d x_d
&=\int_{[-\pi,0]}e^{-ij_dx_d}\tilde{f}(\bm x)d x_d+\int_{[0,\pi]}e^{-ij_dx_d}\tilde{f}(\bm x)d x_d
=\int_{[0,\pi]}(-e^{ij_dx_d}+e^{-ij_dx_d})\tilde{f}(\bm x)d x_d \\
&=-2i\int_{[0,\pi]}\sin(j_dx_d)\tilde{f}(\bm x)d x_d
\end{align*}
So,
\begin{align*}
c_{\bm j}(\tilde{f})
=\frac{-2i}{(2\pi)^d}\int_{[-\pi,\pi]^{d-1}}e^{-i\sum\limits_{k=1}^{d-1} j_kx_k}\int_{[0,\pi]}\sin(j_dx_d)\tilde{f}(\bm x)d\bm x
\end{align*}
By induction, the same process yields,
\begin{align*}
c_{\bm j}(\tilde{f})
&=\frac{(-2i)^d}{(2\pi)^d}\int_{[0,\pi]^{d}}\prod\limits_{k=1}^d\sin(j_kx_k)\tilde{f}(\bm x)d\bm x
=\frac{1}{(i\pi)^d}\int_{[0,\pi]^{d}}\prod\limits_{k=1}^d\sin(j_kx_k)f(\bm x)d\bm x \\
&=\frac{1}{(i\pi)^d}\int_{[0,\pi]^{d}}\prod\limits_{k=1}^d\sin(\textnormal{sign}(j_k)\vert j_k \vert x_k)f(\bm x)d\bm x
=\frac{1}{(i\pi)^d}\int_{[0,\pi]^{d}}\bm\varepsilon(\bm j)\prod\limits_{k=1}^d\sin(\vert j_k \vert x_k)f(\bm x)d\bm x\\
&=\frac{1}{(i\pi)^d}\bm\varepsilon(\bm j)\int_{[0,\pi]^{d}}\left(\frac{\pi}{2}\right)^d e_{\vert\bm j\vert}(\bm x)f(\bm x)d\bm x
=\frac{1}{(2i)^d}\bm\varepsilon(\bm j)\left(f, e_{\vert\bm j\vert}\right)_H
\end{align*}
Moreover, $f$ can be written $\forall \bm x\in\mathbb{R}^d$:
\begin{align*}
f(\bm x)=\sum\limits_{\bm k\in \mathbb{N}^d} (f, e_{\bm k})_He_{\bm k}(\bm x)
\end{align*}
Using Euler's formula, it is quite straightforward to show that $\forall \bm k\in\mathbb{N}^d$:
$$\prod\limits_{l=1}^d\sin(k_lx_l)=\frac{1}{(2i)^d}\sum\limits_{\bm j\in \mathbb{Z}^d : \vert\bm j\vert = \bm k} \bm\varepsilon(\bm j)e^{i\bm j^T\bm x}$$
Therefore,
\begin{align*}
f(\bm x)&=\frac{1}{(i\pi)^d}\sum\limits_{\bm k\in \mathbb{N}^d} (f, e_{\bm k})_H\sum\limits_{\bm j\in \mathbb{Z}^d : \vert\bm j\vert = \bm k} \bm\varepsilon(\bm j)e^{i\bm j^T\bm x}
=\frac{1}{(i\pi)^d} \sum\limits_{\bm j\in \mathbb{Z}^d}\bm\varepsilon(\bm j)(f, e_{\vert\bm j\vert})_H e^{i\bm j^T\bm x} \\
&=\frac{(2i)^d}{(i\pi)^d} \sum\limits_{\bm j\in \mathbb{Z}^d}c_{\bm j} e^{i\bm j^T\bm x}=\left(\frac{2}{\pi}\right)^d \mathcal{S}_F[\tilde{f}](\bm x)
\end{align*}

\end{proof}

It is therefore possible to define the Fourier transform of a function of $f\in L^2([0,\pi]^d)$ as the Fourier transform of its associated $2\pi$-periodic function (of $\mathbb{R}^d$) $\tilde{f}$ defined as in \eqref{transfo_per}. Using this convention, the Fourier transform $\mathcal{F}[f]$ of $f\in L^2([0,\pi]^d)$ is given by:
\begin{equation}
\mathcal{F}[f](\bm\omega)=(2\pi)^d\sum\limits_{\bm j\in\mathbb{Z}^d}c_{\bm j}(\tilde{f})\delta_{\bm j}(\bm\omega)=(-i\pi)^d\sum\limits_{\bm j\in\mathbb{Z}^d}\varepsilon(\bm j)(f, e_{\vert\bm j\vert})_H\delta_{\bm j}(\bm\omega)
\end{equation}

\begin{lemma}
Let $\phi\in L^2([0,\pi]^d)$ such that $\gamma(-\Delta)\phi\in L^2([0,\pi]^d)$. Then,  $\gamma(-\Delta)\phi$ is equal to the restriction to $[0,\pi]^d$ of $\mathcal{F}^{-1}\left[\bm w \mapsto \gamma(\Vert \bm w\Vert^2)\mathcal{F}[\phi](\bm\omega)\right]$, where $\mathcal{F}$ is the Fourier transform operator.
\end{lemma}
\begin{proof}
On one hand, by definition of $\gamma(-\Delta)$, 
$$\gamma(-\Delta)\phi=\sum\limits_{\bm j\in\mathbb{N}^d}g(\lambda_{\bm j})(\phi, e_{\bm j})_H e_{\bm j}$$
where $\lambda_{\bm j}$ and $e_{\bm j}$ are defined in \eqref{eig_lab}.
Then $\forall\bm\omega\in\mathbb{R}^d$,
$$\mathcal{F}\left[\gamma(-\Delta)\phi\right](\bm\omega)
=(-i\pi)^d\sum\limits_{\bm j\in\mathbb{Z}^d}\varepsilon(\bm j)\gamma(\lambda_{\vert\bm j\vert})(\phi, e_{\vert\bm j\vert})_H\delta_{\bm j}(\bm\omega)$$
On the other hand, notice that $\forall\bm\omega\in\mathbb{R}^d$:
\begin{align*}
\gamma(\Vert\bm\omega\Vert^2)\mathcal{F}[\phi](\bm\omega)
&=(-i\pi)^d\sum\limits_{\bm j\in\mathbb{Z}^d}\varepsilon(\bm j)(\phi, e_{\vert\bm j\vert})_H\gamma(\Vert\bm\omega\Vert^2)\delta_{\bm j}(\bm\omega)\\
&=(-i\pi)^d\sum\limits_{\bm j\in\mathbb{Z}^d}\varepsilon(\bm j)(\phi, e_{\vert\bm j\vert})_H\gamma(\Vert\bm j\Vert^2)\delta_{\bm j}(\bm\omega)
=(-i\pi)^d\sum\limits_{\bm j\in\mathbb{Z}^d}\varepsilon(\bm j)(\phi, e_{\vert\bm j\vert})_H\gamma(\lambda_{\vert\bm j\vert})\delta_{\bm j}(\bm\omega)\\
&=\mathcal{F}\left[\gamma(-\Delta)\phi\right](\bm\omega)
\end{align*}
which proves the result.
\end{proof}

This last result can be generalized to more general bounded domains $\mathcal{D}$ or to Riemannian manifolds by defining the Fourier transform on such domains from the decomposition onto the orthonormal basis of eigenfunctions of the Laplacian \citep{adcock2010multivariate}. 

More precisely, the Fourier transform is seen as an application from $L^2(\mathcal{D})$ to $\ell^2(\mathbb{N})$, that associates to each $f\in L^2(\mathcal{D})$ the sequence $\lbrace(f, e_j)_H\rbrace_{j\ge 1}$ of coefficients of the decomposition of $f$ onto the eigenbasis of the Laplacian $\lbrace e_j\rbrace_{j\ge 1}$. The frequency domain $\ell^2(\mathbb{N})$ is a discrete one, indexed by the eigenvalues of the Laplacian: functions on this domain are therefore square-summable sequencies representing the evaluation of a function over the set of admissible frequencies, i.e. the eigenvalues of the Laplacian. From this definition, it is straightforward to check that $\gamma(-\Delta)$ coincides with the operator $\mathcal{F}^{-1}\left[\gamma\mathcal{F}[.]\right]$.


\section{Proof of Theorem \ref{th:err_fem}}
\label{proof:th_err_fem}
Let $Z_{N_h}$ be the random field defined by as the truncation of $Z$ after $N_h$ terms:
\begin{equation}
Z_{N_h}=\sum\limits_{j=1}^{N_h}\gamma(\lambda_j)\tilde{\xi}_j e_j
\end{equation}

Then, from the triangular inequality: 
\begin{equation}
\Vert Z - Z_h \Vert_{L^2(\Omega; H)}\le \Vert Z - Z_{N_h} \Vert_{L^2(\Omega; H)} + \Vert Z_{N_h} - Z_h \Vert_{L^2(\Omega; H)}
\end{equation}

\subsection{Truncation error}
\begin{align*}
\Vert Z - Z_{N_h} \Vert_{L^2(\Omega; H)}^2 =\e\left[\Vert \sum\limits_{j>N_h}\gamma(\lambda_j)\tilde{\xi}_j e_j \Vert_{H}^2\right]=\e\left[\sum\limits_{j>N_h}\gamma(\lambda_j)^2\tilde{\xi}_j^2 \right] =\sum\limits_{j>N_h}\gamma(\lambda_j)^2 
\end{align*}

Then from Assumptions \ref{assum:h} and \ref{assum:gamma}, we have:
\begin{align*}
\Vert Z - Z_{N_h} \Vert_{L^2(\Omega; H)}^2 
&\le  C_\gamma^2 \sum\limits_{j>N_h}  \lambda_j^{-2\beta} 
\end{align*}
And using Assumption \ref{assum:lambda}:
\begin{align*}
\Vert Z - Z_{N_h} \Vert_{L^2(\Omega; H)}^2 
&\le  C_\gamma^2 c_\lambda^{-2\beta}\sum\limits_{j>N_h}  j^{-2\alpha\beta} 
\le \frac{C_\gamma^2 c_\lambda^{-2\beta}}{(2\alpha\beta-1)} \times \frac{1}{N_h^{2\alpha\beta-1}}
\end{align*}

Finally, Assumption \ref{assum:Nh} yields:
\begin{equation}
\Vert Z - Z_{N_h} \Vert_{L^2(\Omega; H)}^2 
\le \frac{C_\gamma^2 c_\lambda^{-2\beta}}{(2\alpha\beta-1)C_{\text{FES}}^{2\alpha\beta-1}} \times h^{\tilde d(2\alpha\beta-1)}
\end{equation}

\subsection{Finite element discretization error}

\begin{align*}
 \Vert Z_{N_h} - Z_h \Vert_{L^2(\Omega; H)} &=\left\Vert \sum\limits_{j=1}^{N_h}\gamma(\lambda_j)\tilde{\xi}_j e_j - \sum\limits_{j=1}^{N_h}\gamma(\lambda_{j,h})\tilde{\xi}_j e_{j,h} \right\Vert_{L^2(\Omega; H)} \\
 & \le \left\Vert \sum\limits_{j=1}^{N_h}\gamma(\lambda_j)\tilde{\xi}_j e_j - \sum\limits_{j=1}^{N_h}\gamma(\lambda_j)\tilde{\xi}_j e_{j,h} \right\Vert_{L^2(\Omega; H)}
+\left\Vert \sum\limits_{j=1}^{N_h}\gamma(\lambda_j)\tilde{\xi}_j e_{j,h}  - \sum\limits_{j=1}^{N_h}\gamma(\lambda_{j,h})\tilde{\xi}_j e_{j,h} \right\Vert_{L^2(\Omega; H)}
\end{align*}
$$=(\text{I})+(\text{II})$$
On one hand,
\begin{align*}
(\text{I})^2 &= \left\Vert \sum\limits_{j=1}^{N_h}\gamma(\lambda_j)\tilde{\xi}_j (e_j-e_{j,h}) \right\Vert_{L^2(\Omega; H)}^2=\e\left[ \left(\sum\limits_{j=1}^{N_h}\gamma(\lambda_j)\tilde{\xi}_j (e_j-e_{j,h}),  \sum\limits_{j=1}^{N_h}\gamma(\lambda_j)\tilde{\xi}_j (e_j-e_{j,h})\right)_H \right] \\
&=  \sum\limits_{j=1}^{N_h}\sum\limits_{k=1}^{N_h}\gamma(\lambda_j)\gamma(\lambda_k)\e\left[\tilde{\xi}_j\tilde{\xi}_k\right]\left( (e_j-e_{j,h}),   (e_k-e_{k,h})\right)_H =\sum\limits_{j=1}^{N_h}\gamma(\lambda_j)^2\Vert e_j-e_{j,h}\Vert_H^2
\end{align*}

So, following Assumption \ref{assum:Lh}, 
\begin{align*}
(\text{I})^2 &\le C_2 \times h^{2s}\sum\limits_{j=1}^{N_h}\gamma(\lambda_j)^2\lambda_j^q 
\end{align*}

Let $J_0$ be the integer defined by $J_0=\left\lceil\left(\frac{R_{\gamma}}{c_\lambda}\right)^{1/\alpha}\right\rceil$. According to Assumptions \ref{assum:h} and \ref{assum:Nh}, $N_h \ge J_0$. Therefore, we write:

\begin{align*}
(\text{I})^2 &\le C_2S_0 \times h^{2s} +C_2 \times h^{2s}\sum\limits_{j=J_0}^{N_h}\gamma(\lambda_j)^2\lambda_j^q
\end{align*}
where $S_0$ the constant defined by $S_0=\sum_{j=1}^{J_0-1}\gamma(\lambda_j)^2\lambda_j^q$. 

Remark then that according to Assumptions \ref{assum:lambda} and \ref{assum:gamma}, $j \le J_0 \Rightarrow \lambda_j \ge R_\gamma$ and therefore,
\begin{align*}
(\text{I})^2 &\le C_2S_0 \times h^{2s} +C_2C_\gamma^2 \times h^{2s}\sum\limits_{j=J_0}^{N_h}\lambda_j^{-2\beta}\lambda_j^q \\
&\le C_2S_0 \times h^{2s} +C_2C_\gamma^2C_\lambda^{q-2\beta} \times h^{2s}\sum\limits_{j=J_0}^{N_h}j^{\alpha(q-2\beta)}\\
&\le C_2S_0 \times h^{2s} +C_2C_\gamma^2C_\lambda^{q-2\beta} \times h^{2s}N_h^{1+\alpha(q-2\beta)}
\end{align*}

So, following Assumption \ref{assum:Nh}, 
\begin{align*}
(\text{I})^2 \le C_2S_0 \times h^{2s} +C_2C_\gamma^2C_\lambda^{q-2\beta}C_{\text{FES}}^{1+\alpha(q-2\beta)} \times h^{2s-\tilde d\alpha q + \tilde d(2\alpha\beta-1)}
\end{align*}

On the other hand,
\begin{align*}
(\text{II})^2 &= \left\Vert \sum\limits_{j=1}^{N_h}(\gamma(\lambda_j)-\gamma(\lambda_{j,h}))\tilde{\xi}_j e_{j,h} \right\Vert_{L^2(\Omega; H)}
=\e\left[\Vert \sum\limits_{j=1}^{N_h}(\gamma(\lambda_j)-\gamma(\lambda_{j,h}))\tilde{\xi}_j e_{j,h} \Vert_{H}^2\right]\\
&=\e\left[\sum\limits_{j=1}^{N_h}(\gamma(\lambda_j)-\gamma(\lambda_{j,h}))^2\tilde{\xi}_j^2 \right] 
=\sum\limits_{j=1}^{N_h}(\gamma(\lambda_j)-\gamma(\lambda_{j,h}))^2
\end{align*}

In particular, using the mean value theorem, for all $1\le j\le N_h$ there exists $l_j\in [\lambda_j, \lambda_{j,h}]$ such that: 
$$\gamma(\lambda_j)-\gamma(\lambda_{j,h})=\gamma'(l_j)(\lambda_{j,h}-\lambda_j)$$
So, using Assumption \ref{assum:hold},
$$\vert\gamma(\lambda_j)-\gamma(\lambda_{j,h})\vert=\vert\gamma'(l_j)\vert\vert\lambda_{j,h}-\lambda_j\vert \le \frac{C_{\text{Deriv}}}{l_j^a}\vert\lambda_{j,h}-\lambda_j\vert\le \frac{C_{\text{Deriv}}}{\lambda_j^a}\vert\lambda_{j,h}-\lambda_j\vert $$
And using Assumptions \ref{assum:lambda} and \ref{assum:Lh},
$$\vert\gamma(\lambda_j)-\gamma(\lambda_{j,h})\vert\le \frac{C_{\text{Deriv}}}{(c_\lambda j^\alpha)^a} C_1h^r(C_\lambda j^\alpha)^q$$
Therefore,
\begin{align*}
(\text{II})^2 &\le \left(C_{\text{Deriv}}c_\lambda^{-a}C_1 C_\lambda^q\right)^2 \times h^{2r}\sum\limits_{j=1}^{N_h} j^{2\alpha(q-a)} 
\le \left(C_{\text{Deriv}}c_\lambda^{-a}C_1 C_\lambda^q\right)^2 \times h^{2r}N_h^{2\alpha(q-a)+1} 
\end{align*}
And using Assumption \ref{assum:Nh}:
\begin{align*}
(\text{II})^2 &\le \left(C_{\text{Deriv}}c_\lambda^{-a}C_1 C_\lambda^q\right)^2C_{\text{FES}}^{2\alpha(q-a)+1} \times h^{2r-2d\alpha(q-a)-\tilde d} 
\end{align*}

\subsection{Total error}

Combining the terms $(\text{I})$ and $(\text{II})$ gives:
\begin{equation*}
\Vert Z_{N_h} - Z_h \Vert_{L^2(\Omega; H)} \le \sqrt{M_1 \times h^{2s} +M_2 h^{2s-\tilde d\alpha q + \tilde d(2\alpha\beta-1)}}+M_3 h^{(r-dq\alpha)+\tilde d(a\alpha-1/2)}
\end{equation*}
where $M_1, M_2, M_3$ are constants independent of $h$.
Using the fact that $h<1$, this last bound can actually be simplified by noticing that all the terms $h^u$ can be bounded by the one with the smallest exponent:
\begin{equation*}
\Vert Z_{N_h} - Z_h \Vert_{L^2(\Omega; H)} \le M h^{\min\left(s, s-dq\alpha/2, \tilde d(\alpha\beta-1/2), r-dq\alpha, \tilde d(\alpha a-1/2) \right)}
\end{equation*}
where $M$ is a constant independent of $h$.

\end{appendices}

\end{document}